\documentclass[10pt,a4paper,twoside]{amsart}

\usepackage{graphicx}
\usepackage{a4wide}

\newtheorem{definition}{Definition}[section]
\newtheorem{theorem}{Theorem}[section]
\newtheorem{lemma}{Lemma}[section]

\newtheorem*{maintheorem*}{Main Theorem}
\allowdisplaybreaks
\numberwithin{equation}{section}

\newcommand{\Kruzkov}{Kru\v{z}kov~}

\newcommand{\norm}[1]{\left\| #1 \right\|}

\newcommand{\eps}{\varepsilon}

\newcommand{\ue}{u_\eps}

\newcommand{\uek}{u_{\eps_k}}
\newcommand{\vek}{v_{\eps_k}}

\newcommand{\Pe}{P_\eps}

\newcommand{\pt}{\partial_t}

\newcommand{\px}{\partial_x }
\newcommand{\pxx}{\partial_{xx}^2}

\renewcommand{\i}{\ifmmode\mathit{\mathchar"7010 }\else\char"10 \fi}
\renewcommand{\j}{\ifmmode\mathit{\mathchar"7011 }\else\char"11 \fi}
\newcommand{\R}{\mathbb{R}}
\newcommand{\N}{\mathbb{N}}

\newcommand{\supp}{\mathrm{supp}\,}
\newcommand{\Hneg}{H_{\mathrm{loc}}^{-1}}
\newcommand{\CL}{\mathcal{L}}
\newcommand{\CLea}{\mathcal{L}}

\newcommand{\sgn}[1]{\mathrm{sign}\left(#1\right)}

\newcommand{\ve}{v_\varepsilon}

{%

\begin{enumerate}}%
{\end{enumerate}}

{%

\begin{enumerate}}%
{\end{enumerate}}

\begin{document}\large

\title[The exp-Rabelo equation]{On the Wellposedness of the exp-Rabelo equation}
\author[G. M. Coclite and L. di Ruvo]{Giuseppe Maria Coclite and Lorenzo di Ruvo}
\address[Giuseppe Maria Coclite and Lorenzo di Ruvo]
{\newline Department of Mathematics,   University of Bari, via E. Orabona 4, 70125 Bari,   Italy}
\email[]{giuseppemaria.coclite@uniba.it, lorenzo.diruvo@uniba.it}
\urladdr{http://www.dm.uniba.it/Members/coclitegm/}
\date{\today}

\thanks{The authors are members of the Gruppo Nazionale per l'Analisi Matematica, la Probabilit\`a e le loro Applicazioni (GNAMPA) of the Istituto Nazionale di Alta Matematica (INdAM)}



\begin{abstract}
The exp-Rabelo equation describes pseudo-spherical surfaces. It is a nonlinear evolution equation.
In this paper  the wellposedness of bounded from above solutions for the initial value problem associated to this equation is studied.
\end{abstract}

\maketitle

%\tableofcontents

\section{Introduction}
\label{sec:intro}
B\"acklund transformations have been useful in the calculation of soliton solutions of certain nonlinear evolution
equations of physical significance \cite{BD, La, Rs1, Rs2} restricted to one space variable $x$ and a time coordinate $t$.
The classical treatment of the surface transformations, which provide the origin of B\"acklund theory, was developed in \cite{G}.
B\"acklund transformations are local geometric transformations, which construct from a given surface of constant Gaussian curvature $-1$
a two parameter family of such surfaces. To find such transformations, one needs to solve a system of compatible
ordinary differential equations \cite{ET}.

In \cite{KCS, KCS1}, the authors used the notion of  differential equation for a function $u(t,x)$ that describes a pseudo-spherical
surface, and they derived some B\"acklund transformations for  nonlinear evolution equations which are the integrability condition $sl(2,R)-$ valued linear problems \cite{CKI, ACKS,EHK, CKSS,Rs2}.

In \cite{KW}, the authors had derived some B\"acklund transformations for nonlinear evolution
equations of the AKNS class. These transformations explicitly express the new solutions in terms of the known solutions of the  nonlinear evolution
equations and corresponding wave functions which are solutions of the associated Ablowitz-Kaup-Newell-Segur (AKNS) system \cite{AKNS,ZS}.

In \cite{KCS2}, the authors used B\"acklund transformations derived in \cite{KCS,KCS1} in the construction of exact soliton solutions for some nonlinear evolution equations describing pseudo-spherical surfaces which are beyond the AKNS class. In particular, they analyzed the following equation \cite{BRT}:
\begin{equation}
\label{eq:ra1}
\px\left(\pt u +\alpha  g(u)\px u +\beta\px u\right)=\gamma g'(u), \quad \alpha,\,\beta,\,\gamma\in\R,
\end{equation}
where $g(u)$ is any solution of the linear ordinary differential equation
\begin{equation}
\label{eq:ra2}
g''(u)+\mu g(u)=\theta, \quad \mu,\,\theta\in\R.
\end{equation}
\eqref{eq:ra1} include the sine-Gordon, sinh-Gordon and Liouville’ equations, in correspondence of $\alpha=0$.

In \cite{RA}, Rabelo proved that the system of the equations \eqref{eq:ra1} and \eqref{eq:ra2} describes pseudo-spherical surfaces and possesses a zero-curvature representation with a parameter.

Let us consider \eqref{eq:ra1}, and assume that $\alpha\neq 0$. In particular, we choose
\begin{equation}
\label{eq:alpha1}
\alpha=-1.
\end{equation}
Taking  $\mu=0,\,\theta=1$, \eqref{eq:ra2} reads
\begin{equation}
\label{eq:g1}
g''(u)=1.
\end{equation}
A solution of \eqref{eq:g1} is
\begin{equation}
\label{eq:solg1}
g(u)=\frac{u^2}{2}.
\end{equation}
Taking $\beta=0,\,\gamma>0$, substituting \eqref{eq:alpha1}, and \eqref{eq:solg1} in \eqref{eq:ra1}, we get
\begin{equation}
\label{eq:she1}
\px\left(\pt u -\frac{1}{6}\px u^3\right)=\gamma u.
\end{equation}
\eqref{eq:she1} was also introduced recently by Sch\"afer and Wayne \cite{SW} as  a model equation describing the propagation of ultra-short light pulses in silica optical fibers.

Integrating \eqref{eq:she1} in  $x$, we gain the integro-differential formulation of  \eqref{eq:she1} (see \cite{SS})
\begin{equation}
\label{eq:SPE-bis}
\pt u -\frac{1}{6} \px u^3=\gamma \int^x u(t,y) dy,
\end{equation}
that is equivalent to
\begin{equation}
\label{eq:integ-1}
\pt u-\frac{1}{6} \px u^3=\gamma P,\qquad \px P=u.
\end{equation}
In \cite{Cd1,CdK,dR}, the authors investigated the well-posedness in classes of
discontinuous functions for \eqref{eq:SPE-bis}, or \eqref{eq:integ-1}. In particular, they proved that \eqref{eq:SPE-bis}, or \eqref{eq:integ-1} admits an unique entropy solution in the sense of the following definition:
\begin{definition}
\label{def:sol-1}
We say that $u\in  L^{\infty}((0,T)\times\R),\ T>0$, is an entropy solution of  \eqref{eq:SPE-bis}, or \eqref{eq:integ-1} if
\begin{itemize}
\item[$i$)] $u$ is a distributional solution of \eqref{eq:SPE-bis} or equivalently of \eqref{eq:integ-1};
\item[$ii$)] for every convex function $\eta\in  C^2(\R)$ the
entropy inequality
\begin{equation}
\label{eq:SPEentropy-1}
\pt \eta(u)+ \px q(u)-\gamma\eta'(u) P\le 0, \qquad     q(u)=-\int^u \frac{\xi^2}{2} \eta'(\xi)\, d\xi,
\end{equation}
holds in the sense of distributions in $(0,\infty)\times\R$.
\end{itemize}
\end{definition}
Here, we consider the case
\begin{equation}
\label{eq:alpha2}
\alpha=1.
\end{equation}
Taking $\mu=-1,\,\theta=0$, \eqref{eq:ra2} reads
\begin{equation}
\label{eq:g2}
g''(u)-g(u)=0.
\end{equation}
A solution of \eqref{eq:g2} is
\begin{equation}
\label{eq:solg2}
g(u)=e^u.
\end{equation}
Taking $\beta=0,\,\gamma=-1$, and substituting \eqref{eq:alpha2}, and \eqref{eq:solg2} in \eqref{eq:ra1}, we get
\begin{equation}
\label{eq:expra}
\px\left(\pt u +\px e^u\right)=- e^u,
\end{equation}
which is known as the exp-Rabelo equation (see \cite{SS2}).

Our aim is to investigate the well-posedness for the initial value problem in classes of discontinuous functions for \eqref{eq:expra}. 
Therefore, we augment \eqref{eq:expra} with the initial datum
\begin{equation}
\label{eq:init}
u(0,x)=u_0(x),
\end{equation}
on which we assume that
\begin{equation}
\label{eq:sup}
\sup u_0 <\infty,\qquad \int_{\R}e^{u_0(x)} dx <\infty.
\end{equation}
Integrating \eqref{eq:expra} in $(0,x)$ we gain the integro-differential formulation of  \eqref{eq:expra} (see \cite{AC,SS2,WSK})
\begin{equation}
\label{eq:expra-u}
\begin{cases}
\pt u+ \px e^u = -\int_{0}^x e^{u(t,y)} dy,&\qquad t>0, \ x\in\R,\\
u(0,x)=u_0(x), &\qquad x\in\R,
\end{cases}
\end{equation}
that is equivalent to
\begin{equation}
\label{eq:integ}
\begin{cases}
\pt u+ \px e^u= -P, &\qquad t>0, \ x\in\R,\\
\px P=e^u, &\qquad t>0, \ x>0,\\
P(t,0)=0,& \qquad t>0,\\
u(0,x)=u_0(x), &\qquad x\in\R.
\end{cases}
\end{equation}
We give the following definition of solution:
\begin{definition}
\label{def:sol-2}
We say that $u$, such that
\begin{equation}
\label{eq:con-ini}
\sup u(t,\cdot)<\infty,\quad \int_{\R} e^{u(t,x)} dx< \infty, \quad t>0,
\end{equation}
is an entropy solution of the initial
value problem \eqref{eq:expra} and  \eqref{eq:init} if
\begin{itemize}
\item[$i$)] $u$ is a distributional solution of \eqref{eq:expra-u} or equivalently of \eqref{eq:integ};
\item[$ii$)] for every convex function $\eta\in  C^2(\R)$ the
entropy inequality
\begin{equation}
\label{eq:SPEentropy-2}
\pt \eta(u)+ \px q(u)+\eta'(u) \int_{0}^x e^{u}dy\le 0, \qquad     q(u)=\int^u e^{\xi} \eta'(\xi)\, d\xi,
\end{equation}
holds in the sense of distributions in $(0,\infty)\times\R$.
\end{itemize}
\end{definition}

The main result of this paper  is the following theorem.
\begin{theorem}
\label{th:main}
Let $T>0$. Assume \eqref{eq:sup}.
The initial value problem
\eqref{eq:expra} and \eqref{eq:init} possesses
an unique entropy solution $u$ in the sense of Definition \ref{def:sol-2}.
Moreover, if $u$ and $w$ are two entropy solutions of  \eqref{eq:expra} and \eqref{eq:init} in the sense of Definition \ref{def:sol-2} the following inequality holds
 \begin{equation}
 \label{eq:stability}
\norm{u(t,\cdot)-w(t,\cdot)}_{L^1(-R,R)}\le  e^{C(T) t}\norm{u(0,\cdot)-w(0,\cdot)}_{L^1(-R-C(T)t,R+C(T)t)},
\end{equation}
for almost every $0<t<T$, $R>0$, and some suitable constant $C(T)>0$.
\end{theorem}
The paper is organized as follows. In Section \ref{sec:vv} we prove several a priori estimates on a vanishing viscosity approximation of \eqref{eq:integ}.
Those play a key role in the proof of our main result, that is given in Section \ref{sec:proof}.

\section{Vanishing viscosity approximation}
\label{sec:vv}
Our existence argument is based on passing to the limit
in a vanishing viscosity approximation of \eqref{eq:integ}.

Fix a small number $\eps>0$, and let $\ue=\ue (t,x)$ be the unique classical solution of the following mixed problem
\begin{equation}
\label{eq:OHepsw}
\begin{cases}
\pt \ue+\px e^{\ue}=-\Pe+ \eps\pxx\left(e^{\ue}\right),&\quad t>0,\ x\in\R,\\
\px\Pe=e^{\ue},&\quad t>0,\ x\in\R,\\
\Pe(t,0)=0,&\quad t>0,\\
\ue(0,x)=u_{\eps,0}(x),&\quad x\in\R,
\end{cases}
\end{equation}
where $u_{\eps,0}$ is $C^\infty(0,\infty)$ approximations of $u_{0}$  such that
\begin{equation}
\begin{split}
\label{eq:u0eps}
& u_{0,\eps}\to u_{0},\quad \text{a.e. and in $L^p_{loc}(\R),\, 1\leq p< \infty$},\\
&\sup u_{\eps,0}\le \sup u_{0}, \quad \eps>0,\\
&\int_{\R}e^{u_{\eps,0}(x)}dx \le \int_{\R} e^{u_{0}(x)} dx,\quad \eps>0.
\end{split}
\end{equation}
Clearly, \eqref{eq:OHepsw} is equivalent to the integro-differential problem
\begin{equation}
\label{eq:OHepswint}
\begin{cases}
\pt \ue+\px e^{\ue}=-\int_{0}^x e^{\ue (t,y)}dy+ \eps\pxx\left(e^{\ue}\right),&\quad t>0,\ x\in\R ,\\
\ue(0,x)=u_{\eps,0}(x),&\quad x\in\R.
\end{cases}
\end{equation}
Observe that, multiplying \eqref{eq:OHepswint} by $e^{\ue(t,x)}$, we have
\begin{equation}
\label{eq:eq-exp}
\pt\left(e^{\ue}\right)+\frac{1}{2}\px\left(e^{2\ue}\right) = -e^{\ue}\int_{0}^x e^{\ue(t,y)}dy +\eps e^{\ue}\pxx\left(e^{\ue}\right).
\end{equation}
Introducing the notation
\begin{equation}
\label{eq:def-di-v}
\ve(t,x)= e^{\ue(t,x)}>0,
\end{equation}
\eqref{eq:eq-exp} reads
\begin{equation}
\label{eq:eq-di-v}
\pt \ve +\frac{1}{2} \px \ve^2= -\ve\int_{0}^x \ve(t,y) dy +\eps \ve\pxx \ve.
\end{equation}
It follows from \eqref{eq:def-di-v} and $\ue(t,\pm\infty)=-\infty$ that
\begin{equation}
\label{eq:v-in-infty}
\ve(t,\pm\infty)=0.
\end{equation}
Moreover, from \eqref{eq:u0eps} and \eqref{eq:def-di-v}, we get
\begin{equation}
\label{eq:v0eps}
\norm{v_{\eps,0}}_{L^{\infty}(\R)}\le e^{\sup u_{0}},\quad \int_{\R}v_{\eps,0}(x)dx \le \int_{\R} e^{u_{0}(x)} dx, \quad \eps>0.
\end{equation}
Let us prove some a priori estimates on $\ve$, and, hence on $\ue$.

\begin{lemma}
\label{lm_l-infty}
We have that
\begin{equation}
\label{eq:l-infty}
\norm{\ve}_{L^{\infty}((0,\infty)\times\R)} \le e^{\sup u_{0}},\quad \eps>0.
\end{equation}
In particular, we get
\begin{equation}
\label{eq:sup1}
\sup\ue(t,\cdot) \le \sup u_{0}, \quad t>0.
\end{equation}
\end{lemma}
\begin{proof}
We begin by observing that, from \eqref{eq:def-di-v} and \eqref{eq:eq-di-v}, we have
\begin{equation}
\pt \ve + \px\left(\frac{\ve^2}{2}\right) -\eps\pxx\left(\frac{\ve^2}{2}\right) \le 0.
\end{equation}
Therefore, a supersolution of \eqref{eq:eq-di-v} satisfies the following ordinary differential equation
\begin{equation*}
\frac{dz}{dt}=0, \quad z(0)= e^{\sup u_{0}},
\end{equation*}
that is
\begin{equation}
z(t)= e^{\sup u_{0}}.
\end{equation}
It follows from the comparison principle for parabolic equation and \eqref{eq:def-di-v} that
\begin{equation}
\label{eq:sup2}
0< \ve(t,x) \le e^{\sup u_{0}},
\end{equation}
which gives \eqref{eq:l-infty}.

Finally, \eqref{eq:sup1} follows from \eqref{eq:def-di-v} and \eqref{eq:sup2}.
\end{proof}

\begin{lemma}\label{lm:l-n}
Let $\alpha\ge 0$. For each $t>0$, we have
\begin{equation}
\label{eq:l-n}
\begin{split}
\int_{\R} \ve^{\alpha+1} &+\eps(\alpha+1)^2\int_{0}^{t}\!\!\!\int_{\R}\ve^{\alpha}(\px\ve)^2dsdx\\
 &+(\alpha+1)\int_{0}^t\!\!\!\int_{\R}\ve^{\alpha+1}\left(\int_{0}^{x}\ve dy\right)dsdx  \le \left(e^{\sup u_{0}}\right)^{\alpha}\int_{\R}e^{u_{0}(x)} dx.
\end{split}
\end{equation}
In particular, we get
\begin{equation}
\begin{split}
\label{eq:e-l-n}
\int_{\R} e^{(\alpha+1)\ue(t,x)} dx &\le \left(e^{\sup u_{0}}\right)^{\alpha}\int_{\R}e^{u_{0}(x)} dx,\\
\eps(\alpha+1)^2\int_{0}^{t}\!\!\!\int_{\R}e^{\alpha\ue(t,x)}\left(\px e^{\ue(t,x)}\right)^2dsdx &\le \left(e^{\sup u_{0}}\right)^{\alpha}\int_{\R}e^{u_{0}(x)} dx,\\
(\alpha+1)\int_{0}^t\!\!\!\int_{\R}e^{(\alpha+1)\ue(t,x)}\left(\int_{0}^{x}e^{\ue(t,x)} dy\right)dsdx  &\le \left(e^{\sup u_{0}}\right)^{\alpha}\int_{\R}e^{u_{0}(x)} dx.
\end{split}
\end{equation}
\end{lemma}

\begin{proof}
Multiplying \eqref{eq:eq-di-v} by $\ve^{\alpha}$, we have
\begin{equation*}
\ve^{\alpha} \pt\ve + \ve^{\alpha+1}\px\ve = -\ve^{\alpha+1}\int_{0}^{x} \ve dy+\eps\ve^{\alpha+1}\pxx\ve.
\end{equation*}
It follows from \eqref{eq:def-di-v}, \eqref{eq:eq-di-v} and an integration on $\R$ that
\begin{align*}
\frac{1}{\alpha+1}\frac{d}{dt}\int_{\R}\ve^{\alpha+1}=&\int_{\R}\ve^{\alpha}\pt\ve\\
=& \eps\int_{\R}\ve^{\alpha+1}\pxx\ve dx - \int_{\R} \ve^{\alpha+1}\px\ve -\int_{\R}\ve^{\alpha+1}\left(\int_{0}^{x} \ve dy\right) dx\\
=& -\eps(\alpha+1)\int_{\R}\ve^{\alpha}(\px\ve)^2 dx  -\int_{\R}\ve^{\alpha+1}\left(\int_{0}^{x} \ve dy\right) dx,
\end{align*}
that is,
\begin{equation}
\label{eq:23}
\frac{d}{dt}\int_{\R}\ve^{\alpha+1}+\eps(\alpha+1)^2\int_{\R}\ve^{\alpha}(\px\ve)^2 dx+(\alpha+1)\int_{\R}\ve^{\alpha+1}\left(\int_{0}^{x} \ve dy\right) dx= 0.
\end{equation}
An integration on $(0,t)$ gives
\begin{equation}
\label{eq:24}
\begin{split}
\int_{\R}\ve^{\alpha+1}dx &+ \eps(\alpha+1)^2\int_{0}^{t}\!\!\!\int_{\R}\ve^{\alpha}(\px\ve)^2 dsdx\\
&+(\alpha+1)\int_{0}^{t}\!\!\!\int_{\R}\ve^{\alpha+1}\left(\int_{0}^{x} \ve dy\right) dx=\int_{\R}v_{\eps,0}^{\alpha+1} dx.
\end{split}
\end{equation}
From \eqref{eq:def-di-v} and \eqref{eq:v0eps},
\begin{equation}
\label{eq:234}
\int_{\R}v_{\eps,0}^{\alpha+1}dx \le \norm{v_{\eps,0}}^{\alpha}_{L^{\infty}((0,\infty)\times\R)}\int_{\R} v_{\eps,0}dx \le \left(e^{\sup u_{0}}\right)^{\alpha}\int_{\R}e^{u_{0}(x)} dx.
\end{equation}
Therefore, \eqref{eq:24} and \eqref{eq:234} give \eqref{eq:l-n}.

Finally, \eqref{eq:e-l-n} follows from \eqref{eq:eq-di-v} and \eqref{eq:l-n}
\end{proof}

\section{Proof of Theorem \ref{th:main}}
\label{sec:proof}
This section is devoted to the proof of Theorem \ref{th:main}. We begin with the following result

\begin{lemma}\label{lm:conv-u}
Let $T>0$. There exists a subsequence
$\{\vek\}_{k\in\N}$ of $\{\ve\}_{\eps>0}$
and a limit function $  v\in L^{\infty}((0,\infty)\times\R)$
such that
\begin{equation}\label{eq:convu}
    \textrm{$\vek \to v$ a.e.~and in $L^{p}_{loc}((0,\infty)\times\R)$, $1\le p<\infty$}.
\end{equation}
In particular, we have
\begin{equation}
\label{eq:conv-u}
\textrm{$\uek \to \log v= u$ a.e.~and in $L^{p}_{loc}((0,\infty)\times\R)$, $1\le p<\infty$}
\end{equation}
\end{lemma}
\begin{proof}
Let $\eta:\R\to\R$ be any convex $C^2$ entropy function, and
$q:\R\to\R$ be the corresponding entropy
flux defined by $q'(v)=v\eta'(v)$.
By multiplying \eqref{eq:eq-di-v} with
$\eta'(\ve)$ and using the chain rule, we get
\begin{align*}
    \pt  \eta(\ve)+\px q(\ve)
    =&\underbrace{\eps \px\left( \eta'(\ve)\ve\px\ve\right)}_{=:\CLea_{1,\eps}}
    \, \underbrace{-\eps \eta''(\ve)\ve\left(\px  \ue\right)^2}_{=: \CLea_{2,\eps}}\\
    \,& \underbrace{-\eps \eta'(\ve)\left(\px  \ve\right)^2}_{=: \CLea_{3,\eps}}
     \, \underbrace{-\eta'(\ve)\ve\int_{0}^{x} \ve dy}_{=: \CLea_{4,\eps}},
\end{align*}
where  $\CLea_{1,\eps}$, $\CLea_{2,\eps}$, $\CLea_{3,\eps}$, $\CLea_{4,\eps}$ are distributions.
Let us show that
\begin{equation*}
\label{eq:H1}
\textrm{$\CLea_{1,\eps}\to 0$ in $H^{-1}((0,T)\times\R)$, $T>0$.}
\end{equation*}
By Lemmas \ref{lm_l-infty} and \ref{lm:l-n} in correspondence of $\alpha=0$,
\begin{align*}
\norm{\eps\eta'(\ve)\ve\px\ve}^2_{L^2((0,T)\times\R)}&\le\eps ^2\norm{\eta'}^2_{L^{\infty}(I)}\norm{\ve}_{L^{\infty}((0,\infty)\times\R)}\int_{0}^{T}\norm{\px\ve(s,\cdot)}^2_{L^2(\R)}ds\\
&\le\eps\norm{\eta'}^2_{L^{\infty}(T)}e^{\sup u_{0}}\int_{\R}e^{u_{0}(x)} dx\to 0,
\end{align*}
where
\begin{equation*}
I=\left(0, e^{\sup u_{0}}\right).
\end{equation*}
We claim that
\begin{equation*}
\label{eq:L1}
\textrm{$\{\CLea_{2,\eps}\}_{\eps>0}$ is uniformly bounded in $L^1((0,T)\times\R)$, $T>0$}.
\end{equation*}
Again by Lemmas \ref{lm_l-infty} and \ref{lm:l-n} in correspondence of $\alpha=0$,
\begin{align*}
\norm{\eps\eta''(\ve)\ve(\px\ve)^2}_{L^1((0,T)\times\R)}&\le
\norm{\eta''}_{L^{\infty}(I)}\norm{\ve}_{L^{\infty}((0,\infty)\times\R)}\eps
\int_{0}^{T}\norm{\px\ve(s,\cdot)}^2_{L^2(\R)}ds\\
&\le \norm{\eta''}_{L^{\infty}(I)}e^{\sup u_{0}}\int_{\R}e^{u_{0}(x)} dx.
\end{align*}
We have that
\begin{equation*}
\textrm{$\{\CL_{3,\eps}\}_{\eps>0}$ is uniformly bounded in $L^1((0,T)\times\R)$, $T>0$.}
\end{equation*}
Again by Lemmas \ref{lm_l-infty} and \ref{lm:l-n} in correspondence of $\alpha=0$,
\begin{align*}
\norm{\eps\eta'(\ve)(\px\ve)^2}_{L^1((0,T\times\R)}\le &\norm{\eta'}_{L^{\infty}(I)}\eps\int_{0}^{T}\norm{\px\ve(s,\cdot)}^2_{L^2(\R)}ds\\
\le& \norm{\eta'}_{L^{\infty}(I)}\int_{\R}e^{u_{0}(x)} dx.
\end{align*}
We claim that
\begin{equation*}
\textrm{$\{\CL_{4,\eps}\}_{\eps>0}$ is uniformly bounded in $L^1((0,T)\times\R)$, $T>0$.}
\end{equation*}
Again by Lemmas \ref{lm_l-infty} and \ref{lm:l-n} in correspondence of $\alpha=0$,
\begin{align*}
\norm{\eta'(\ve)\ve\int_{0}^{x} \ve dy}_{L^1((0,T)\times\R)}\le & \norm{\eta'}_{L^{\infty}(I)}\int_{0}^{T}\!\!\!\int_{\R}\ve\left(\int_{0}^{x}\ve dy\right)dsdx\\
\le &  \norm{\eta'}_{L^{\infty}(I)}\int_{\R}e^{u_{0}(x)} dx.
\end{align*}
Therefore, Murat's lemma \cite{Murat:Hneg} implies that
\begin{equation}
\label{eq:GMC1}
    \text{$\left\{  \pt  \eta(\ve)+\px q(\ve)\right\}_{\eps>0}$
    lies in a compact subset of $\Hneg((0,T)\times\R)$.}
\end{equation}
The $L^{\infty}$ bound stated in Lemma \ref{lm_l-infty}, \eqref{eq:GMC1}, and the
 Tartar's compensated compactness method \cite{TartarI} give the existence of a subsequence
$\{\vek\}_{k\in\N}$ and a limit function $  v\in L^{\infty}((0,\infty)\times\R),$
such that \eqref{eq:convu} holds.

\eqref{eq:conv-u} follows from \eqref{eq:def-di-v} and \eqref{eq:convu}.
\end{proof}

To prove  Theorem \ref{th:main}, we consider the following definition.
\begin{definition}
A pair of functions $(\eta, q)$ is called an  entropy--entropy flux pair if $\eta :\R\to\R$ is a $C^2$ function and $q :\R\to\R$ is defined by
\begin{equation*}
q(u)=\int_{0}^{u} \eta'(\xi)f'(\xi)d\xi.
\end{equation*}
An entropy-entropy flux pair $(\eta,\, q)$ is called  convex/compactly supported if, in addition, $\eta$ is convex/compactly supported.
\end{definition}

\begin{proof}[Proof of Theorem \ref{th:main}]
We begin by proving that $u$, defined in \eqref{eq:conv-u}, is an entropy solution of \eqref{eq:expra-u}, or \eqref{eq:integ} in the sense of Definition \ref{def:sol-2}. Let $\phi\in C^{\infty}(\R^2)$ be a positive text function with a support, and let us consider a compactly supported entropy--entropy flux pair $(\eta, q)$. We have to prove
\begin{equation}
\label{eq:3433}
\begin{split}
\int_{0}^{\infty}\!\!\!\int_{\R}(\eta(u)\pt\phi+q(u)\px\phi)dtdx &- \int_{0}^{\infty}\!\!\!\int_{\R}\eta'(u)\left(\int_{0}^{x}e^{u} dy \right)dtdx\\
 &+ \int_{\R}\eta(u_{0}(x))\phi(0,x)dx \ge 0.
\end{split}
\end{equation}
Multiplying \eqref{eq:OHepsw} by $\eta'(\ue)$, we have
\begin{equation*}
\pt \eta(\uek)+\px q(\uek) +\eta'(\uek)\int_{0}^{x} e^{\uek} dy = \eps_{k}\eta'(\uek)\pxx\left(e^{\uek}\right).
\end{equation*}
Since
\begin{align*}
\eps_{k}\eta'(\uek)\pxx\left(e^{\uek}\right)=&\px\left(\eps_{k}\eta'(\uek)\px\left (e^{\uek}\right)\right) -\eps_{k}\eta''(\uek)\px\left(e^{\uek}\right)\px\uek\\
=& \px\left(\eps_{k}\eta'(\uek)\px\left( e^{\uek}\right)\right)-\eps_{k} \eta''(\uek) e^{\uek}(\px\uek)^2,
\end{align*}
we have
\begin{equation}
\label{eq:ph1}
\begin{split}
\pt \eta(\uek)&+\px q(\uek) +\eta'(\uek)\int_{0}^{x} e^{\uek} dy\\
 = &\px\left(\eps_{k}\eta'(\uek)\px\left( e^{\uek}\right)\right)-\eps_{k} \eta''(\uek) e^{\uek}(\px\uek)^2\\
\le& \px\left(\eps_{k}\eta'(\uek)\px\left( e^{\uek}\right)\right).
\end{split}
\end{equation}
Multiplying \eqref{eq:ph1} by $\phi$, an integration on $(0,\infty)\times\R$ gives
\begin{equation}
\label{eq:phi2}
\begin{split}
\int_{0}^{\infty}\!\!\!\int_{\R}(\eta(\uek)\pt\phi+q(\uek)\px\phi)dtdx &- \int_{0}^{\infty}\!\!\!\int_{\R}\eta'(\uek)\left(\int_{0}^{x}e^{\uek} dy \right)dtdx\\
& + \int_{\R}\eta(u_{\eps_{k}, 0}(x))\phi(0,x)dx \\
&+\eps_{k}\int_{0}^{\infty}\!\!\!\int_{\R}\eta'(\uek)\px\left( e^{\uek}\right)\px\phi dtdx\ge 0.
\end{split}
\end{equation}
Let us show that
\begin{equation}
\label{eq:phi3}
\eps_{k}\int_{0}^{\infty}\!\!\!\int_{\R}\eta'(\uek)\px\left( e^{\uek}\right)\px\phi dtdx\to 0.
\end{equation}
Fix $T>0$. From \eqref{eq:l-n} in correspondence of $\alpha=0$, and the H\"older inequality,
\begin{align*}
&\eps_{k}\int_{0}^{\infty}\!\!\!\int_{\R}\eta'(\uek)\px\left( e^{\uek}\right)\px\phi dtdx\\
&\quad \le \eps_{k} \int_{0}^{\infty}\!\!\!\int_{\R} \vert \eta'(\uek) \vert \left\vert \px\left( e^{\uek}\right) \right\vert\vert\px\phi\vert dtdx\\
&\quad \le \eps_{k} \norm{\eta'}_{L^{\infty}((0,\infty)\times\R)}\norm{\px\left( e^{\uek}\right)}_{L^2(\supp(\px\phi))}\norm{\px\phi}_{L^2(\supp(\px\phi))}\\
&\quad \le \eps_{k} \norm{\eta'}_{L^{\infty}((0,\infty)\times\R)}\norm{\px\left( e^{\uek}\right)}_{L^2((0,T)\times\R)}\norm{\px\phi}_{L^2((0,T)\times\R)}\\
&\quad \le \sqrt{\eps_{k}}\norm{\eta'}_{L^{\infty}((0,\infty)\times\R)}\norm{\px\phi}_{L^2((0,T)\times\R)}\left(\int_{\R}e^{u_{0}(x)} dx\right)^{\frac{1}{2}}\to 0,
\end{align*}
that is \eqref{eq:phi3}.

Therefore, \eqref{eq:phi2} follows from \eqref{eq:u0eps}, \eqref{eq:conv-u}, \eqref{eq:phi2} and \eqref{eq:phi3}.

Let us prove that $u(t,x)$ is unique and \eqref{eq:stability} holds.
Assume that $u(t,x),\, w(t,x)$ satisfy
\begin{equation}
\begin{split}
\label{eq:178}
&\sup u(t,\cdot)\le \sup u_{0}, \quad \sup w(t,\cdot) \le \sup w_{0},\quad t>0,\\
&\int_{\R} e^{u_{0}(x)}dx < \infty,\quad  \int_{\R} e^{w_{0}(x)} dx <\infty,
\end{split}
\end{equation}
two entropy solutions of \eqref{eq:expra-u}, or \eqref{eq:integ}. Due to \eqref{eq:178}, we have
\begin{equation}
\label{eq:179}
\left\vert e^{u}-e^{w}\right\vert \le C_{0}\vert u -w\vert,
\end{equation}
where
\begin{equation*}
C_{0} = \sup_{\R}\{ e^{\sup u_{0}} + e^{\sup w_{0}}\}.
\end{equation*}
Arguing as in \cite{CdK, dR, K}, we can prove that
\begin{equation*}
\pt(\vert u - w \vert ) + \px\left((e^{u} - e^{w}\right)\sgn{u-w}) +\sgn{u-w}\int_{0}^{x}\left(e^{u}-e^{w}\right)dy \leq 0
\end{equation*}
holds in sense of distributions in $(0,\infty)\times\R$, and
\begin{equation}
\label{eq:1002}
\begin{split}
&\norm{u(t,\cdot)-w(t,\cdot)}_{I(t)}\\
&\quad \leq\norm{u_{0}-w_{0}}_{I(0)}-\int_{0}^{t}\!\!\int_{I(s)}\sgn{u-w}\left(\int_{0}^{x}\left(e^{u}-e^{w}\right)dy\right) dsdx, \quad 0<t<T,
\end{split}
\end{equation}
where
\begin{equation*}
I(s)=[-R-C_{0}(t-s),R+C_{0}(t-s)].
\end{equation*}
Due to \eqref{eq:179},
\begin{equation}
\label{eq:1003}
\begin{split}
&-\int_{0}^{t}\!\!\int_{I(s)}\sgn{u-v} \left(\int_{0}^{x}\left(e^{u}-e^{w}\right)dy\right)dsdx\\
&\quad\le\int_{0}^{t}\!\!\int_{I(s)}\left\vert \left(\int_{0}^{x}\left\vert e^{u}-e^{w}\right\vert dy\right)\right\vert dsdx\\
&\quad \le C_{0}\int_{0}^{t}\!\!\int_{I(s)}\Big(\Big\vert\int_{I(s)}\vert u-v\vert dy\Big\vert\Big)dsdx\\
&\quad =C_{0} \int_{0}^{t} \vert I(s)\vert\norm{u(s,\cdot)-v(s,\cdot)}_{L^{1}(I(s))}ds.
\end{split}
\end{equation}
Moreover,
\begin{equation}
\label{eq:1005}
\vert I(s) \vert=2R+2C_{0}(t-s)\le 2R+2C_{0}t\le 2R+2C_{0}T.
\end{equation}
We consider the following continuous function:
\begin{equation}
\label{eq:defG}
G(t)=\norm{u(t,\cdot)-v(t,\cdot)}_{L^{1}(I(t))}, \quad t\ge 0.
\end{equation}
It follows from \eqref{eq:1002}, \eqref{eq:1003}, \eqref{eq:1005} and \eqref{eq:defG} that
\begin{equation*}
G(t)\le G(0)+C(T)\int_{0}^{t}G(s)ds,
\end{equation*}
where $C(T)=2R+2C_{0}T$. The Gronwall inequality and  \eqref{eq:defG} give
\begin{equation*}
\norm{u(t,\cdot)-v(t,\cdot)}_{L^{1}(-R,R)}\le e^{C(T)t}\norm{u_{0}-v_{0}}_{L^{1}(-R-C_{0}t,R+C_{0}t)},
\end{equation*}
that is  \eqref{eq:stability}.
\end{proof}


\begin{thebibliography}{50}
\bibitem{AC}
{\sc M. J. Ablowitz and P. A. Clarkson.}
\newblock Solitons, Nonlinear Evolution Equations and Inverse Scattering.
\newblock {\em Cambridge University Press}, Cambridge, vol. 149, UK, 1991.

\bibitem{AKNS}
{\sc M.J. Ablowitz, D.J. Kaup,  A.C. Newell, and H. Segur.}
\newblock The inverse scattering transform-Fourier analysis for nonlinear problems.
\newblock {\em Stud. Appl. Math.}, 53:249–-315, 1974.

\bibitem{BRT}
{\sc R. Beals,  M. Rabelo, and K. Tenenblat.}
\newblock B\"acklund transformations and inverse scattering solutions for some pseudospherical surface equations.
\newblock{\em Stud. Appl. Math.}, 81:125–-151, 1989.



\bibitem{Cd1}
{\sc G.~M. Coclite and L. di Ruvo.}
\newblock Wellposedness results for the short pulse equation.
\newblock To appear on {\em Z. Angew. Math. Phys.}

\bibitem{CdK}
{\sc G.~M. Coclite, L. di Ruvo, and K.~H. Karlsen.}
\newblock Some wellposedness results for the Ostrovsky-Hunter equation.
\newblock Hyperbolic conservation laws and related analysis with applications, 143-159, Springer Proc. Math. Stat., 49, Springer, Heidelberg, 2014.

\bibitem{dR}
{\sc L. di Ruvo.}
\newblock Discontinuous solutions for the Ostrovsky--Hunter equation and two phase flows.
\newblock {\em Phd Thesis, University of Bari}, 2013.
\newblock{www.dm.uniba.it/home/dottorato/dottorato/tesi/}.

\bibitem{BD}
{\sc R.K. Dodd  and R.K. Bullough.}
\newblock B\"acklund transformations for the AKNS inverse method.
\newblock{\em Phys. Lett. A}, 62:70–-74, 1977.


\bibitem{ET}
{\sc L.P. Eisenhart.}
\newblock {A Treatise on the Differential Geometry of Curves and Surfaces}.
\newblock {\em Ginn \& Co., 1909, Dover}, NewYork, 1960.



\bibitem{G}
{\sc E. Goursat}.
\newblock Le Probl\`eme de B\"acklund.
\newblock{ \em M\'emorial des Sciences Mathématiques}, Fasc. VI, Gauthier-Villars, Paris, 1925.

\bibitem{ACKS}
{\sc A.H. Khater, D.K. Callebaut, A.A. Abdalla,   and S.M. Sayed.}
\newblock Exact solutions for self-dualYang–Mills equations.
\newblock {\em Chaos Solitons Fractals}, 10:1309–-1320, 1999.

\bibitem{CKI}
{\sc A.H. Khater, D.K. Callebaut, and R.S. Ibrahim.}
\newblock B\"acklund transformations and Painlev\'e analysis: exact solutions for the unstable nonlinear
Schr\"odinger equation modelling electron-beam plasma.
\newblock{\em Phys. Plasmas}, 5:395–-400, 1998.

\bibitem{KCS}
{\sc A.H. Khater, D.K. Callebaut, and S.M. Sayed.}
\newblock Conservation laws for some nonlinear evolution equations which describe pseudospherical surfaces.
\newblock{\em J. Geom. Phys.}, 51:332–-352, 2004.

\bibitem{KCS1}
{\sc   A.H. Khater, D.K. Callebaut, and  S.M. Sayed.}
\newblock B\"acklund transformations for some nonlinear evolution equations which describe pseudospherical surfaces.
\newblock Submitted.

\bibitem{KCS2}
{\sc   A.H. Khater, D.K. Callebaut, and  S.M. Sayed.}
\newblock Exact solutions for some nonlinear evolution equations which describe pseudo-spherical surfaces
\newblock{\em J. Comp. and Appl. Math.}, 189:387--411, 2006.

\bibitem{EHK}
{\sc A.H. Khater,  M.A. Helal, and O.H. El-Kalaawy.}
\newblock Two new classes of exact solutions for the KdV equation via B\"acklund transformations
\newblock {\em Choas Solitons Fractals}, 8:1901–-1909, 1997.


\bibitem{CKSS}
{\sc A.H. Khater, A.M. Shehata, D.K. Callebaut, and S.M. Sayed.}
\newblock Self-dual solutions for SU(2) and SU(3) gauge fields one Euclidean space, Internat.
\newblock{\em J. Theoret. Phys.}, 43:151–-159, 2004.

\bibitem{KW}
{\sc K. Konno and  M. Wadati.}
\newblock Simple derivation of B\"acklund transformation from Riccati Form of inverse method.
\newblock{\em  Progr. Theoret. Phys.}, 53:1652–-1656, 1975.


\bibitem{K}
{\sc S. ~N. \Kruzkov}
\newblock First order quasilinear equations with several independent variables.
\newblock {\em Mat. Sb. (N.S.)}, 81(123), 28:228--255, 1970.

\bibitem{La}
{\sc M.G. Lamb.}
\newblock B\"acklund transformations for certain nonlinear evolution equations.
\newblock {\em J. Math. Phys.}, 15:2157–-2165, 1974.

%\bibitem{LPS}
%{\sc Y. Liu, D. Pelinovsky and A. Sakovich.}
%\newblock Wave breaking in the short-pulse equation.
%\newblock {\em Dynamics of PDE}, 6:291--310, 2009.

\bibitem{Murat:Hneg}
{\sc F.~Murat.}
\newblock L'injection du c\^one positif de ${H}\sp{-1}$\ dans ${W}\sp{-1,\,q}$\  est compacte pour tout $q<2$.
\newblock {\em J. Math. Pures Appl. (9)}, 60(3):309--322, 1981.



\bibitem{RA}
{\sc M. Rabelo}
\newblock On equations which describe pseudospherical surfaces.
\newblock{\em Stud. Appl. Math}, 81:221–-248, 1989.

\bibitem{Rs1}
{\sc C. Rogers and W.K. Schief.}
\newblock B\"acklund and Darboux transformations, in Geometry and Modern Applications in Soliton Theory.
\newblock {\em Cambridge Texts in Applied Mathematics, Cambridge University Press}, Cambridge, 2002.

\bibitem{Rs2}
{\sc C. Rogers and W.K. Schief.}
\newblock B\"acklund Transformations and their Applications.
\newblock {\em Academic Press}, New York, 1982.



\bibitem{SS}
{\sc A. Sakovich and S. Sakovich.}
\newblock The short pulse equation is integrable.
\newblock {\em J. Phys. Soc. Jpn.}, 74:239--241, 2005.

%\bibitem{SS1}
%{\sc A. Sakovich  and S. Sakovich}.
%\newblock{Solitary wave solutions of the short pulse equation}.
%\newblock {\em J. Phys. Soc. Jpn.} 39:361-367, 2006.

\bibitem{SS2}
{\sc A. Sakovich and S. Sakovich}
\newblock On the trasformations of the Rabelo equations.
\newblock{\em SIGMA 3}, 8 pages, 2007.


\bibitem{SW}
{\sc T.~Sch\"afer, and C.E. Wayne.}
\newblock Propagation of ultra-short optical pulses in cubic nonlinear media.
\newblock {\em Physica D}, 196:90--105, 2004.

\bibitem{TartarI}
{\sc L.~Tartar.}
\newblock Compensated compactness and applications to partial differential   equations.
\newblock In {\em Nonlinear analysis and mechanics: Heriot-Watt Symposium, Vol.  IV},
pages 136--212. Pitman, Boston, Mass., 1979.

\bibitem{WSK}
{\sc M. Wadati, H. Sanuki, and K. Konno.}
\newblock Relationships among inverse method, B\"acklund transformation and an infinite number of conservation laws.
\newblock {\em Progress of Theoretical Physics}, vol. 53,  419–-436, 1975.


\bibitem{ZS}
{\sc V.E. Zakharov and A.B. Shabat.}
\newblock Exact theory of two-dimensional self-focusing and one-dimensional self-modulation of waves in nonlinear media.
\newblock {\em Sov. Phys. JETP} 34:62–-69, 1972.


\end{thebibliography}
\end{document}